\documentclass[11pt, twoside, leqno, french]{amsart}
\usepackage[T1]{fontenc}
\usepackage{babel}
\usepackage{amsmath,amsthm,amssymb}
\usepackage{hyperref}

\usepackage{tikz-cd}
\usetikzlibrary{arrows}
\usepackage{enumerate}

\theoremstyle{plain} 
\newtheorem{theorem}{Théorème}[section]
\newtheorem*{theorem*}{Théorème}
\newtheorem{proposition}[theorem]{Proposition}
\newtheorem*{proposition*}{Proposition}

\newtheorem{lemme}[theorem]{Lemme}
\newtheorem*{lemme*}{Lemme}
\newtheorem{scolie}[theorem]{Scolie}
\newtheorem{corollaire}[theorem]{Corollaire}
\newtheorem{conjecture}[theorem]{Conjecture}

\newtheorem*{remarque*}{Remarque}
\newtheorem{inegalite}{Inégalité}

\numberwithin{equation}{section}
\numberwithin{figure}{section}

\newcommand\ag{{\bf a }}
\newcommand\bg{{\bf b }}

\newcommand\xg{{\bf x }}

\newcommand\alphag{{\boldsymbol\alpha}}

\newcommand\deltag{{\boldsymbol\delta}}

\newcommand\lambdag{{\boldsymbol\lambda}}

\newcommand\zetag{{\boldsymbol\zeta}}

\newcommand\C{{\mathbb C}}
\newcommand\K{{K}}%{{\mathbb K}}
\newcommand\N{{\mathbb N}}
\renewcommand\P{{\mathbb P}}
\newcommand\Q{{\mathbb Q}}
\newcommand\R{{\mathbb R}} 
\newcommand\Z{{\mathbb Z}}

\newcommand{\D}{\mathcal{D}}

\newcommand{\opcit}{{\it op. cit.}}
\newcommand{\cf}{{\it cf}}
\newcommand{\ie}{{\it i.e.}}

\DeclareMathOperator{\codim}{codim}

\DeclareMathOperator{\Gal}{Gal}
\DeclareMathOperator{\pgcd}{pgcd}

\DeclareMathOperator{\Supp}{Supp}
\DeclareMathOperator{\Vol}{Vol}

\newcommand{\Gm}{{\mathbb G}_{\rm m}}

\newcommand{\Qb}{{\overline{\mathbb Q}}}
\newcommand{\h}{{\widehat{h}}}
\newcommand{\hP}{{h_{\P_n}}}
\newcommand{\muess}{\hat{\mu}^{{\sf ess}}}

\author[F. AMOROSO, N. H. ANDRIAMANDRATOMANANA et  D. SIMON]{Francesco AMOROSO$^{1}$, Njaka Harilala ANDRIAMANDRATOMANANA$^{2}$ et Denis SIMON$^{2}$}
\title{Sur une conjecture de Schinzel.}

\begin{document}
\maketitle
\begin{center}
{\small\it
$^{(1)}$Dipartimento di Matematica\\
Università di Torino\\
Via Carlo Alberto, 10 - 10123 Torino, Italy\\[0.3cm]
$^{(2)}$Laboratoire de mathématiques Nicolas Oresme, CNRS UMR 6139\\
Université de Caen Normandie, Campus II, BP 5186\\
14032 Caen Cedex, France\\[0.5cm]}
\end{center}

\begin{abstract}
Nous donnons une nouvelle preuve d’une conjecture de Schinzel sur l’intersection d’une sous--variété de codimension au moins 2 dans une puissance du groupe multiplicatif avec un tore de dimension 1. La preuve repose sur un théorème géométrique de Bézout de P. Philippon et sur des minorations pour la hauteur du premier auteur, S. David et E. Viada. Elle donne pour la première fois un résultat explicite, en fonction de la hauteur et du degré de la variété. Il s'inspire d'un énoncé similaire de S. Checcoli, F. Veneziano et E. Viada dans un produit de courbes elliptiques.\\ 

We give a  new proof of a conjecture of Schinzel on the intersection of a subvariety of codimension at least 2 in a power of the multiplicative group with a torus of dimension 1. The proof rests on a geometric Bézout’s theorem of P. Philippon and on lower bounds for the height of the first author, S. David and E. Viada. It gives for the first time an explicit result, depending on the height and degree of the variety. It is inspired on a similar statement on products of elliptic curves, by S. Checcoli, F. Veneziano and E. Viada.
\end{abstract}

\smallskip
\noindent \textbf{\small 2020 Mathematical Subject Classification:} 11J99, 11G30, 11G50.

\bigskip
\noindent \textbf{\small Keywords.} Multiplicative group, polynomials, heights. 
\bigskip

\hfill{\sl In memory of Andrzej Schinzel}

\section{Introduction}
La conjecture de Schinzel à laquelle se réfère le titre de cet article est la suivante~:
\begin{conjecture}[Schinzel, 1965]
\label{conj:schinzel}
Soient $F,G\in \Z[x_1,\dots,x_n]$ premiers entre eux. Soit ensuite $\ag:=(a_1,\dots,a_n)\in\Z^n$ et soit $\xi\in\Qb^*$ différent d'une racine de l'unité. Supposons que
$$
F(\xi^{a_1},\dots,\xi^{a_n}) = G(\xi^{a_1},\dots,\xi^{a_n})=0.
$$
Alors, il existe un vecteur non nul $\bg\in \Z^n$ orthogonal à $\ag$ et de norme $L_\infty$ satisfaisant~:
\begin{equation*}
\Vert \bg \Vert_{\infty} \leq B(F,G),
\end{equation*}
où $B(F,G)>0$ dépend seulement de $F$ et $G$.
\end{conjecture}

A. Schinzel a énoncé cette conjecture dans~\cite[Conjecture, p.~3]{Schinzel1965} et en a donné une preuve dans~\cite[Theorem 1, p.~47]{Schinzel1989} dans le cas particulier $n\leq 3$. On pourra aussi se référer à~\cite[Conjecture 1, p.~279 et Theorem 45, p.~144]{schinzel2000polynomials}. La conjecture~\ref{conj:schinzel} a ensuite été démontrée en toute généralité par E. Bombieri et U. Zannier~\cite[Appendix]{schinzel2000polynomials} et, dans une forme plus générale et avec une méthode partiellement différente, par E. Bombieri, D. Masser et U. Zannier~\cite[Theorem 1.6]{bombieri2007anomalous}.
Ces preuves sont effectives mais pas explicites (voir la discussion après~\cite[Theorem 2]{filaseta2008irreducibility}).

On présente ici une nouvelle démonstration en explicitant la dépendance en le degré et en la taille des coefficient de $F$ et $G$. Notre méthode s'inspire de l'approche de S. Checcoli, F. Veneziano et E. Viada dans~\cite{checcoli2014torsion}, où les auteurs ont démontré un analogue de la conjecture~\ref{conj:schinzel} dans le cas de produit de courbes elliptiques. Notre démarche utilise le théorème de Bézout arithmétique de P.~Philippon et une version fonctorielle du théorème de Dobrowolski généralisé. 

Étant donné un polynôme de Laurent $F\in \Z[x_1^{\pm1},\ldots,x_n^{\pm1}]$, on note $\deg(F)$ son degré (\ie{} le degré de la clôture de Zariski dans $\P_n$ de l'hypersurface de $\Gm^n$ d'équation $F=0$) et $\Vert F\Vert_1$ la norme $L_1$ du vecteur de ses coefficients.

\begin{theorem}
\label{theo:main}
Soient $F_1,\dots,F_s\in\Z[x_1^{\pm1},\ldots,x_n^{\pm1}]$. Notons 
$$
D:=\max_{1\leq i\leq s}\deg(F_i), \quad\hbox{et}\quad  h_1:=\max_{1\leq i\leq s}\log\Vert F_i\Vert_{1},
$$
et $\tilde{h}_1:=h_1+ (n+13)\log (n+1)D$. Soient également $\ag\in\Z^n\setminus\{\mathbf{0}\}$ et $\xi\in \Qb^*$ différent d'une racine de l'unité. 

On suppose qu'il existe une composante irréductible $W$ de codimension $k\geq 2$ de la sous--variété de $\Gm^n$ définie par $F_1=\dots=F_s=0$ et qui passe par le point $(\xi^{a_1},\dots,\xi^{a_n})$. Il existe alors un vecteur non nul $\bg\in \Z^n$ orthogonal à $\ag$ et de norme $L_2$ satisfaisant~:
\begin{equation}
\label{eq:maj_b1}
\Vert \bg \Vert_2 \leq (4n)^{17n}\,\tilde{h}_1 D^{k-1}\log(\tilde{h}_1D^k)^{2n}.
\end{equation}
\end{theorem}

On remarquera que la dépendance quasi--linéaire en la taille des coefficients dans le théorème~\ref{theo:main} s'accorde bien avec l'exemple suivant dû à Schinzel~:
\begin{multline*}
n=s=2, \quad F_1(x_1,x_2)=x_1-2, \; F_2(x_1,x_2)=x_2-2^a\\
\text{ et }\quad \xi=2, \; (a_1,a_2)=(1,a),
\end{multline*}
et également avec~\cite[Theorem 2.3]{asz2017},
où les auteurs montrent une borne pour $\Vert \bg \Vert_\infty$, indépendante de $h_1$ si cette dernière quantité est suffisamment petite par rapport à $\Vert \ag \Vert_\infty$.\\

M. Filaseta, A. Granville et A. Schinzel~\cite[Theorem B]{filaseta2008irreducibility} ont montré à l'aide de l'ex--conjecture~\ref{conj:schinzel} qu'il existe un algorithme qui permet de calculer le $\pgcd$ de deux polynômes $f$, $g\in\Z[x]$ de degré au plus $d\geq 2$ en $O_{f,g}(\log d)$ opérations binaires, sous l'hypothèse qu'au moins l'un des $f$ et $g$ n'a pas de facteur cyclotomique. La constante implicite dans $O_{f,g}$ dépend, effectivement mais pas explicitement, du nombre de coefficients de $f$ et $g$ et de leur taille. Ensuite, le premier auteur, L. Leroux et M. Sombra~\cite[Theorem 4.3]{amoroso2015overdetermined} ont précisé ce résultat, en montrant qu'il existe un algorithme qui permet de déterminer la partie non cyclotomique du pgcd de $f$ et $g$ en au plus $O_{f,g}(\log d)$ opérations binaires, même si $f$ et $g$ ont des facteurs cyclotomiques. La constante implicite dans $O_{f,g}$ dépend à nouveau de façon effective mais pas explicite du nombre de coefficients des polynômes et de leur taille. %Le théorème~\ref{theo:main} permet en principe d'expliciter cette dépendance.
Le théorème~\ref{theo:main} permet d'expliciter cette dépendance. En reprenant l'algorithme D de \cite{leroux2011these} (voir aussi~\cite{amoroso2015overdetermined}, algorithme 4.2), qui est à son tour une version plus précise de l'algorithme décrit dans le théorème B de~\cite{filaseta2008irreducibility}, on peut obtenir~:
$$
O_{f,g}(\log d)=\widetilde{O}\left((2h_1)^{(cN)^{N}}(1+\log d)\right),
$$
où $c\geq 1$ est une constante absolue, $N$ est le nombre de coefficients non nuls de $f$ et $g$, et $h_1$ une majoration pour la taille logarithmique de leur coefficients.
Cette estimation doublement exponentielle en $N$ pourrait suggérer de revenir sur l'algorithme D, car toutes les opération en $O_{f,g}(1)$ dans cet algorithme ont désormais un coût explicite en fonction de $N$ et $h_1$. L'aspect doublement exponentiel de cette complexité peut s'expliquer par le caractère exhaustif de la recherche de vecteurs de norme bornée, où les bornes utilisées sont obtenues récursivement à chaque étape comme puissance des bornes à l'étape précédente.

\subsection*{Plan de l'article}
Dans le paragraphe~\ref{section:Gmn}, après avoir rappelé quelques faits standards sur la géométrie dans une puissance d'un groupe multiplicatif, nous énoncerons des résultats de géométrie des nombres qui nous seront utiles dans la suite. Dans le paragraphe~\ref{section:hauteurs}, on rappellera les définitions de hauteur projective et normalisée et du minimum essentiel. Le paragraphe~\ref{section:bezout} sera consacré au rappel du théorème de Bézout de P. Philippon. Dans le paragraphe~\ref{section:minimum} on présentera une preuve d'une version fonctorielle du théorème de Dobrowolski en dimension supérieure de S. David et du premier auteur~\cite{amoroso1999probleme}. Cette minoration se déduit du résultat principale de~\cite{amoroso2012} à l'aide d'un lemme de transfert de G. Rémond. Nous prouverons le théorème~\ref{theo:main} dans le paragraphe~\ref{section:proof}.

\subsection*{Remerciements}
L'énoncé et la preuve du lemme~\ref{lemme:remond} sont de Ga\"el Rémond, qui les a transmis au premier auteur. Le premier auteur remercie également Patrice Philippon pour plusieurs clarifications concernant sa version du théorème de Bézout arithmétique et de l'article~\cite{philippon-sombra}, Francesco Veneziano pour des échanges à propos de la stratégie de la preuve de~\cite{checcoli2014torsion}.
Il remercie également Luca de Feo (IBM Zürich) et  Sina Schaeffler (IBM / ETH Zürich).\\

F. Amoroso et D. Simon remercient le projet IEA (International Emerging Action) PAriAIPP (Problèmes sur l'Arithmétique et l'Algèbre des Petits Points) du CNRS pour le soutien financier. F. Amoroso est membre du réseau INdAM GNSAGA.

\section{Généralités sur $\Gm^n$, géométrie des nombres}
\label{section:Gmn}
Pour les notations et les rappels ci--dessous, le lecteur pourra se rapporter à~\cite[chapitre IV, section 2.2]{Zan2024}. 

Nous travaillons dans une puissance $\Gm^n$ du groupe multiplicatif. Pour $\lambdag\in\Z^n$ et $\xg=(x_1,\ldots,x_n)$, on pose $\xg^\lambdag=x_1^{\lambda_1}\cdots x_n^{\lambda_n}$. Pour un entier $l$, on note aussi $[l]\colon\Gm^n\to\Gm^n$ la \og multiplication\fg{} par $l$, \ie{} le morphisme $\xg\mapsto\xg^l:=(x_1^l ,\ldots,x_n^l)$. 

Par sous--variété $V$ de $\Gm^n$, nous entendons un fermé de Zariski définie sur $\Qb$. On dit que $V$ est irréductible si sa fermeture est géométriquement irréductible. De même, on dit que $V$ est irréductible sur un corps $K\subseteq\Qb$ si sa fermeture de Zariski est irréductible sur $K$.

Par sous--groupe algébrique de $\Gm^n$ nous entendons une sous--variété algébrique fermée stable sous les opérations de groupe. Un sous--groupe algébrique irréductible est appelé (sous--)tore. Nous appelons sous--variété de torsion de $\Gm^n$, un translaté d'un sous--tore par un point de torsion. 
Tout sous--groupe algébrique est une union finie de sous--variétés de torsion.
 
Soit $\Lambda\subseteq\Z^n$ un sous--groupe de rang $k$. Alors 
$$
H_\Lambda
:=\{\xg\in\Gm^n,\; \forall\lambdag\in\Lambda,\;\xg^\lambdag=1\}
$$
est un sous--groupe algébrique de dimension $n-k$. De plus, $\Lambda\mapsto H_\Lambda$ est une bijection entre sous--groupes de $\Z^n$ et sous--groupes algébriques de $\Gm^n$. Le sous--groupe $\Lambda$ est saturé (\ie{} si $k\lambdag\in\Lambda$ avec $k\geq1$ entier et $\lambdag\in\Z^n$, alors $\lambdag\in\Lambda$) si et seulement si $H_\Lambda$ est un tore.

On fixe le plongement 
\begin{align*}
\Gm^n&\hookrightarrow\P_n\\
(x_1,\ldots,x_n)&\to(1:x_1:\cdots:x_n)
\end{align*}
Le degré d'une sous--variété est alors le degré de la fermeture de Zariski de son image. On note $\Qb[x_1^{\pm1},\ldots,x_n^{\pm1}]$ l'anneau des fonctions régulières sur $\Gm^n$. Par degré d'un polynôme de Laurent $F\in \Qb[x_1^{\pm1},\ldots,x_n^{\pm1}]$ on entend donc le degré de l'hypersurface d'équation $F=0$. 

Le lemme suivant, dû à M. Laurent~\cite[Lemme 3]{laurent1984equations} (voir aussi \cite[Lemma~4]{schmidt1996heights}), permet de déterminer des générateurs du sous--groupe de $\Z^n$ associé à un sous--groupe algébrique contenu dans une sous--variété de $\Gm^n$. Pour ${\mathcal A}\subset\Z^n$ on note $\D({\mathcal A}):=\{\lambdag_1-\lambdag_2\;\vert\;\lambdag_1, \lambdag_2\in {\mathcal A}\}$.

\begin{lemme}
\label{lemme:schmidt}
Soient $F_1,\dots,F_s\in\Qb[x_1^{\pm1},\ldots,x_n^{\pm1}]$ et soit $H$ un sous--groupe algébrique de $\Gm^n$ contenu dans la sous--variété de $\Gm^n$ définie par $F_1=\dots=F_s=0$, maximal par cette propriété. Il existe alors un sous--groupe $\Lambda$ de $\Z^n$ engendré par des vecteurs dans $\D(\cup_{i=1}^s \Supp(F_i))$ tel que $H=H_\Lambda$.
\end{lemme}

Soit $\Lambda$ un sous--groupe de $\Z^n$. On note $\Vol(\Lambda)$ le (co)volume de $\Lambda$ défini comme étant le volume (pour la mesure de Lebesgue associée) d'une maille fondamentale de $\Lambda$ vu comme un réseau du sous--espace $\R\Lambda$ de $\R^n$ engendré par $\Lambda$. 

\begin{lemme}
\label{lem:geom-nombres}
Soit $\Lambda$ un sous--groupe de $\Z^n$ de rang $k$ et $H_\Lambda\subseteq\Gm^n$ le sous--groupe algébrique associé, de dimension $d:=n-k$. 
On note $\Delta$ le maximum des valeurs absolues des déterminants $k\times k$ d'une base de $\Lambda$. Alors~:
\begin{enumerate}[1)]
\item $\Delta\leq\deg(H_\Lambda)\leq\binom{n}{k}\Delta$~;
\item $\binom{n}{k}^{-1/2}\deg(H_\Lambda)\leq\Vol(\Lambda)\leq\binom{n}{k}^{1/2}\deg(H_\Lambda)$
\end{enumerate}
\end{lemme}
\begin{proof}
Pour $p=1,2,\ldots\infty$, notons $\Vol_p(\Lambda)$ la norme $L_p$ du vecteur des déterminants $k\times k$ d'une base de $\Lambda$. Donc $\Vol_\infty(\Lambda)=\Delta$  et, d'après la formule de Cauchy--Binet, $\Vol_2(\Lambda)=\Vol(\Lambda)$. De plus, 
$$
\Vol_1(\Lambda)\leq\binom{n}{k}\Vol_\infty(\Lambda)\quad\hbox{ et }\quad\Vol_2(\Lambda)\leq\binom{n}{k}^{1/2}\Vol_\infty(\Lambda)
$$ 
et, par Cauchy--Schwarz, $\Vol_1(\Lambda)\leq\binom{n}{k}^{1/2}\Vol_2(\Lambda)$. Par~\cite{philippon-sombra}, p.312, avant--dernière formule et dernière phrase qui précèdent la Proposition 4.2, on a~:
$$
\Vol_\infty(\Lambda)\leq\deg(H_\Lambda)\leq\Vol_1(\Lambda),
$$
Les inégalités du lemme suivent immédiatement.
\end{proof}

Soit à nouveau $\Lambda$ un sous--groupe de $\Z^n$ de rang $k$. Pour $1\leq i \leq k$, on note $\lambda_i(\Lambda)$ le $i$--ème minimum successif de $\Lambda$~:
\begin{equation}
\label{eq:minima}
\lambda_i(\Lambda) = \inf\max_{1\leq j\leq i}\Vert \bg_j\Vert_2,
\end{equation}
où le infimum est pris sur l'ensemble des familles libres $\{\bg_1,\dots,\bg_i\}$ de $\Lambda$ et où  $\Vert\star\Vert_2$ désigne la norme $L_2$. En d'autres termes, $\lambda_i(\Lambda)$ est le rayon de la plus petite boule euclidienne dans $\R^n$
qui contient $i$ vecteurs de $\Lambda$ linéairement indépendants. 

À l'aide du théorème de Minkowski et de l'algorithme LLL, on déduit du point 2) du lemme~\ref{lem:geom-nombres} le résultat suivant~:
\begin{lemme}
\label{lem:minkowski}
Soient $n\geq 2$ et $\ag\in\Z^n\setminus\{\mathbf{0}\}$. Étant donnée une base $B=\{\bg_1,\dots,\bg_{n-1}\}$ de $\ag^\perp$, pour $i=1,\ldots,n-1$ on note $\Lambda_{B,i}$ le sous--groupe primitif engendré par $\bg_1,\dots,\bg_i$.
On désigne également par $V_i=\pi^{i/2}\Gamma(1+i/2)^{-1}$ le volume de la boule euclidienne unité en dimension $i$.
\begin{enumerate}[1)]
\item Soit $B$ une base de $\ag^\perp$. Alors $2^{-i}V_i\binom{n}{i}^{-1/2}\prod_{j=1}^i \lambda_j(\ag^\perp)\leq\deg(H_{\Lambda_{B,i}})$ pour $i=1,\ldots,n-1$.
\item Il existe une base $B$ de $\ag^\perp$ avec $\deg(H_{\Lambda_{B,i}})\leq  \binom{n}{i}^{1/2} \prod_{j=1}^i \lambda_j(\ag^\perp)$ pour $i=1,\ldots,n-1$.
\item
Soit $B$ une base réduite de $\ag^\perp$ au sens de \text{LLL} (\cf~\cite {lenstra1982factoring}). Alors on a~:\\ 
$\deg(H_{\Lambda_{B,i}})\leq\binom{n}{i}^{1/2}2^{i(n-2)/2}\prod_{j=1}^{i}\lambda_j(\ag^\perp)$ pour $i=1,\ldots,n-1$.
\end{enumerate}
\end{lemme}

\begin{proof}
Montrons 1). Soit $B$ une base de $\ag^\perp$. Comme $\Lambda_{B,i} \subseteq \ag^\perp$, on a $\lambda_j(\ag^\perp) \leq \lambda_j(\Lambda_{B,i})$ pour tout $j\in\{1,\dots,i\}$. Donc, d'après le second théorème de Minkowski~\cite[Theorem 16]{siegel1989lectures}, 
$$
\prod_{j=1}^i\lambda_j(\ag^\perp) \leq\prod_{j=1}^i\lambda_j(\Lambda_{B,i}) \leq 2^iV_i^{-1}\Vol(\Lambda_{B,i}).
$$
On obtient l'inégalité souhaitée en combinant avec la majoration de $\Vol(\Lambda_{B,i})$ du lemme~\ref{lem:geom-nombres}, point 2), où $\Lambda=\Lambda_{B,i}$.\\

Montrons maintenant 2). Soient $\tilde\bg_1,\dots,\tilde\bg_{n-1}$ des éléments linéairement indépendants de $\ag^\perp$ tels que
$\Vert \tilde\bg_i \Vert_2 \leq \lambda_i(\ag^\perp)$. 
Par une transformation triangulaire, on construit une base $B:=\{\bg_1,\dots,\bg_{n-1}\}$ de $\ag^\perp$ telle que, pour $i=1,\ldots,n-1$ on ait~: 
$\langle\tilde\bg_1,\dots,\tilde\bg_i\rangle_\Z \subseteq \langle\bg_1,\dots,\bg_i\rangle_\Z$. En utilisant l'inégalité de Hadamard on obtient~:
$$
\Vol(\Lambda_{B,i}) \leq \Vol (\langle\tilde\bg_1,\dots,\tilde\bg_i\rangle_\Z) \leq \prod_{j=1}^i \Vert \tilde\bg_j \Vert_2
\leq \prod_{j=1}^i \lambda_j(\ag^\perp).
$$
On en déduit l'inégalité souhaitée en combinant avec la minoration de $\Vol(\Lambda_{B,i})$ du lemme~\ref{lem:geom-nombres}, point 2), où $\Lambda=\Lambda_{B,i}$.\\

Montrons enfin la troisième partie du lemme. Soit $\{\bg_1,\dots,\bg_{n-1}\}$ une base réduite de $\ag^\perp$ au sens de \text{LLL}. D'après~\cite[Theorem 9]{nguyen2009hermite} (en prenant $\delta = 3/4$), on a~:
$$
\Vert \bg_j\Vert_2 \leq 2^{(n-2)/2}\lambda_j(\ag^\perp) \text{ pour }j\in\{1,\dots,n-1\}.
$$
Comme auparavant, on en déduit l'inégalité souhaitée en combinant avec la minoration de $\Vol(\Lambda_{B,i})$ du lemme~\ref{lem:geom-nombres}, point 2).\\
\end{proof}

Les points 1) et 2) de ce lemme suffisent pour montrer notre résultat principal. En revanche, avoir une majoration comme celle du point 3) 
en fonction du $i$--ème minimum successif d'une base obtenue algorithmiquement pourra être utile pour une version plus performante de l'algorithme pour le calcul du pgcd de deux polynômes lacunaires.

\section{Hauteurs}
\label{section:hauteurs}
Rappelons la notion de hauteur de Weil dans l'espace projectif $\P_n$. Soit $\alphag:=(\alpha_0:\alpha_1:\dots:\alpha_n)\in \P_n(\Qb)$. On choisit un corps de nombres qui contient les coordonnées projectives $\alpha_0,\alpha_1,\dots,\alpha_n$ de $\alphag$. On pose~:
\begin{equation}
\label{def:hauteur}
h(\alphag) := \frac{1}{[\K:\Q]}\sum_v n_v\log\vert\alphag\vert_v
\end{equation}
où la somme porte sur les places $v$ de $\K$ (finies et infinies),
$n_v:=[\K_v:\Q_v]$ est le degré local
et $\vert\alphag\vert_v:=\max(\vert\alpha_0\vert_v,\dots,\vert\alpha_n\vert_v)$.
Cette definition ne dépend ni du choix de $\K$ ni du choix des coordonnées projectives de $\alphag$.

Soit $Y$ une sous--variété géométriquement irréductible de $\P_n$.
Dans~\cite{philippon-ha1},~\cite{philippon-ha2} et~\cite{philippon-ha3},
P.~Philippon définit une notion de hauteur de $Y$, notée $\hP(Y)$, à l'aide de la théorie des formes éliminantes.
La définition de hauteur d'une variété s'étend par additivité
à un cycle de Chow $S=\sum_i [l_i]Y_i$ ($Y_i$ irréductibles)
à coefficients $l_i\in\N$, par~: $\hP(S):=\sum_i l_i\hP(Y_i)$.

Nous remarquons que dans~\cite{philippon-ha3} l'auteur considère des sous--variétés projectives
définies sur un corps de nombres $\K$ et qui sont implicitement supposées $\K$-irréductibles
(au moins on peut se réduire à ce cas en considérant le cycle des conjugués galoisiens de chaque variété).
Ici, nous utiliserons le mot {\sl sous--variété de $\P_n$} pour se référer à une intersection ensembliste
de sous--variétés géométriquement irréductibles de $\P_n$.
Comme on le fait usuellement, nous noterons aussi par $Y\cap W$ l'intersection ensembliste de $Y$ et $W$.
Soient $I$ et $J$ les idéaux homogènes définissant $Y$ et $W$.
On peut alors identifier cette intersection au cycle $ZR(I+J)$
supporté sur les composantes irréductibles de l'intersection et avec coefficients~$1$
(voir~\cite[p.~347]{philippon-ha3}).
Remarquons en revanche que, dans \opcit, le cycle $Y\cap W=Z(I+J)$ n'est en général pas réduit
et peut même avoir des composantes immergées.

Pour un point $\alphag=(\alpha_0:\alpha_1:\dots:\alpha_n)\in \P_n$,
la hauteur de la sous--variété projective $Y=\{\alphag\}$ coïncide avec la hauteur $L_2$ de $\alphag$,
qui est définie en prenant dans~\eqref{def:hauteur}
$$ \Vert\alphag\Vert_v:=(\vert\alpha_0\vert_v^2+\dots+\vert\alpha_n\vert_v^2)^{1/2}$$
aux places archimédiennes au lieu de $\vert\alphag\vert_v$.
Comme pour les points, on dispose également d'une formule explicite
pour la hauteur $\hP(Y)$ d'une hypersurface $Y$.
Soit $\K$ un corps de nombres et soit $F \in \K[x_1,\ldots,x_n]$.
En suivant P.~Philippon
(\cite[p.~346]{philippon-ha3}, en particulier le dernier paragraphe en prenant $p=1$ et $m_1=m$),
on définit une hauteur de $F$ par 
\begin{equation}
\label{def:eqhF}
h(F) := \frac{1}{[\K:\Q]}\sum_{v}[\K_v:\Q_v]\log \widetilde{M}_v(F) + \frac12\deg(F)\sum_{i=1}^{n}\frac{1}i
\end{equation}
où $\widetilde{M}_v(F)$ est le maximum des valeurs absolues $v$--adiques des coefficients de $F$ si $v$ est finie
et $\widetilde{M}_v(F)=\widetilde{M}(\sigma_v(F))$ si $v$ est infinie
associée au plongement $\sigma_v$ de $\K$ dans $\C$,
où, pour un polynôme homogène $P\in\C[x_0,\dots,x_n]$, 
\begin{equation}
\label{def:eqM}
\log \widetilde{M}(P) := \int_{S_{n+1}}\log\vert P\vert \sigma_{n}.
\end{equation}
Ici $\sigma_{n}$ désigne la mesure invariante de masse totale $1$ sur la sphère $S_{n+1}$ de rayon~$1$ de $\C^n$.
Soit maintenant $Z$ une hypersurface de $\P_n$ définie par l'équation $F = 0$, où $F\in\Qb[x_0,\dots,x_n]$ est homogène.
On a alors (\cite[p.~347]{philippon-ha3})
$$
\hP(Z) = h(F) + \frac12\deg(F)\sum_{i=1}^{n-1}\sum_{j=1}^i\frac{1}{j}.
$$
En particulier, si $F\in\Z[x_0,\dots,x_n]$, en tenant compte de la définition~\eqref{def:eqhF}, on a $h(F)\leq \log\Vert F\Vert_1+ \frac12\deg(F)\sum_{i=1}^{n}\frac{1}i$, et donc
\begin{equation}
\label{def:eqhZ}
\hP(Z)\leq \log\Vert F\Vert_1+ \frac12\deg(F)\sum_{i=1}^n\sum_{j=1}^i\frac{1}{j}.
\end{equation} 

Soit $Y$ une sous--variété irréductible de $\Gm^n$. D'après~\cite{philippon-ha3}, on définit la hauteur normalisée de $Y$
(par rapport au plongement $\Gm^n\hookrightarrow \P_n$) par~:
\begin{eqnarray}
\h(Y):=\lim_{m\rightarrow +\infty}\frac{\hP([m]Y)\deg(Y)}{m\deg([m]Y)}
\end{eqnarray}
où on a noté $\hP(Y)$ la hauteur de la clôture de Zariski de l'image de $Y$ dans $\P_n$. On peut vérifier que la hauteur $\h(\alphag)$ de $\alphag\in\Gm^n$ coïncide avec la hauteur de Weil de l'image de $\alphag$ dans $\P_n$. On étend ensuite la definition à une variété définie par linéarité, en prenant la somme des hauteurs normalisées des composantes irréductibles.

D'après~\cite[Proposition 2.1]{david1999minorations}, on a la relation suivante entre $\h(Y)$ et $\hP(Y)$~:
\begin{eqnarray}
\label{eq:hP-hGmn} 
\vert \h(Y) - \hP(Y) \vert \leq \tfrac{7}{2} \left(\dim Y +1\right)\log(n+1)\deg(Y).
\end{eqnarray}

On définit ensuite le minimum essentiel de $Y$, noté $\muess(Y)$, par~:
\begin{equation*}
\muess(Y) :=\inf\left\{\theta>0 \mid \overline{Y(\theta)}= Y \right\}
\end{equation*}
où $Y(\theta) := \{\alphag\in Y(\Qb)\mid \h(\alphag)\leq \theta\}$.
Le minimum essentiel de $Y$ est donc le seuil de la hauteur à partir duquel les points de $Y$ deviennent denses dans~$Y$. 

La hauteur normalisée et le minimum essentiel sont liés par des inégalités de Zhang~\cite[Theorem 5.2]{zhang1995positive}~: 
\begin{eqnarray}
\label{eq:Zhang}
\frac{\h(Y)}{(\dim(Y)+1)\deg(Y)}
\leq \widehat{\mu}^\textrm{ess}(Y)
\leq \frac{\h(Y)}{\deg(Y)}.
\end{eqnarray}

Rappelons (\cite{zhang1991positive}) qu'une variété géométriquement irréductible est de hauteur normalisée nulle (et donc de minimum essentiel nul par les inégalités qui précèdent) si et seulement si c'est une variété de torsion.

\section{Théorème de Bézout}
\label{section:bezout}
Nous utiliserons la version suivante du théorème de Bézout (géométrique et arithmétique) de P.~Philippon (\cite[théorème~3]{philippon-ha3}). 

\begin{theorem}
\label{theo:bezoutarithm}
Soient $Y$ et $W$ deux sous--variétés de $\P_n$ définies sur un corps de nombres $\K$ et $\K$--irréductibles.
Soient $X_1,\dots, X_g$ les composantes irréductibles de l'intersection ensembliste $Y\cap W$.

Alors on a
\begin{align*}
&\sum_{i=1}^g\deg(X_i)\leq\deg(Y)\deg(W),\\
&\sum_{i=1}^g\hP(X_i)\leq\hP(Y)\deg(W) + \hP(W)\deg(Y)+c\deg(W)\deg(Y),
\end{align*}
où
$$
c:=\frac{\log(2)}2(\codim(Y)+\codim(W))+\frac12\sum_{i=0}^{\dim Y}\sum_{j=0}^{\dim W}\frac1{i+j+1}. 
$$
\end{theorem}

\begin{proof}
On considère le cycle intersection $Y \cdot W$, son degré et sa hauteur
(voir \cite[section B, p.~353]{philippon-ha3}).
Toute composante irréductible isolée du cycle $Y \cap W$ apparaît comme composante isolée
du cycle intersection $Y \cdot W$ (\cf{}. \opcit{}, paragraphe avant l'énoncé du théorème 3),
et donc, par définition du degré et de la hauteur du cycle intersection,
$\sum_{i=1}^g\deg(X_i)\leq\deg(Y \cdot W)$ et $\hP(\sum_{i=1}^g\deg(X_i))\leq \hP(Y \cdot W)$.
On applique alors~\cite[théorème 3]{philippon-ha3}.

On pourra également se référer à~\cite[Theorem 3, p.~455]{habegger2008intersecting},
où cependant la valeur de la constante $c$ est moins précise ($c=3n^2$).
\end{proof}

Le corollaire suivant permet de majorer la hauteur d'une variété en fonction de la hauteur et du degré de ses générateurs.

\begin{corollaire}
\label{coro:bezoutarithm}
Soient $F_1,\dots,F_s\in\Z[x_0,x_1,\ldots,x_n]$ des polynômes homogènes définissant des hypersurfaces $Z_i$ de $\P_n$. Notons 
$$
D:=\max_{1\leq i\leq s}\deg(F_i), \quad\hbox{et}\quad  h_1:=\max_{1\leq i\leq s}\log\Vert F_i\Vert_{1}.
$$
Soit $\K\subseteq\Qb$ un corps et $W$ une composante $\K$--irréductible de codimension $k$ de $Z_1\cap\dots\cap Z_l$. Alors on a
$$ 
\deg(W) \leq D^k\quad\hbox{et}\quad\hP(W) \leq n h_1D^{k-1}+n\left(\sum_{i=1}^{n}\sum_{j=1}^i\frac{1}{j}\right)D^k.
$$
\end{corollaire}

\begin{proof}
Soit $Z_i$ l'hypersurface de $\P_n$ définie par $F_i=0$. On applique~\cite[Corollaire 5, p.~357]{philippon-ha3} avec $S=\P_n$, $\deltag=D$ et $\eta=h$,
en tenant compte des notations introduites au milieu de \opcit{} p.~347.
Notons cependant que dans \opcit{} $Z_i$ est supposée définie sur un corps de nombres $\K$ et $\K$--irréductible.
On peut se réduire à cette situation par linéarité et en négligeant les éventuelles multiplicités.
On a alors, en tenant compte du fait que
$\deg(\P_n)=1$ et $\hP(\P_n)= \frac{1}{2}\sum_{i=1}^{n}\sum_{j=1}^i\frac{1}{j}$
(\cf{}~\cite{philippon-ha3}, avant--dernier paragraphe, p. 346),
$$
d(S_l;D)\leq d(\P_n;D)= D^n
$$
et
$$
h(S_l;D)\leq h(\P_n;D) +nh\, d(\P_n;D) = \frac{1}{2}\left(\sum_{i=1}^{n}\sum_{j=1}^i\frac{1}{j}\right)D^{n+1} +nhD^n.
$$
Les composantes irréductibles $W_i$ de $Z_1\cap \dots \cap Z_l$
sont des composantes irréductibles du cycle intersection $S_l$.
En particulier, en négligeant les multiplicités d'intersection
et les composantes de codimension différente de $\codim(W)$, on a~:
$$
d(S_l;D)\geq\sum_i \deg(W_i)D^{\deg(W_i)}\geq \deg(W)D^{\dim(W)}
$$
et
$$
h(S_l;D)\geq\sum_i \hP(W_i)D^{\dim(W)+1}\geq \hP(W)D^{\dim(W)+1}.
$$
En comparant les majorations et les minorations de $d(S_l;D)$ et $h(S_l;D)$ on en déduit la majoration annoncée de $\deg(W)$ et~:
$$ 
\hP(W) \leq n hD^{k-1}+\frac12n\left(\sum_{i=1}^{n}\sum_{j=1}^i\frac{1}{j}\right)D^k,
$$
où $h := \max_{1\leq i\leq s} h(Z_i)$. En tenant compte de la majoration~\eqref{def:eqhZ} pour la hauteur projective d'une hypersurface, on a~:
$$
h\leq
\max_{1\leq i\leq s} \left(\log\Vert F_i\Vert_1 + \frac12\deg(F_i)\sum_{i=1}^n\sum_{j=1}^i\frac{1}{j}\right)
\leq h_1+ \frac12D\sum_{i=1}^{n}\sum_{j=1}^i\frac{1}{j}, 
$$
d'où la majoration annoncée pour $\hP(W)$.
\end{proof}

\section{Minimum essentiel}
\label{section:minimum}
%\begin{center}
%par Ga\"el Rémond\\[0.2cm]
%{\small Institut Fourier, CNRS UMR5582 \& Université Grenoble Alpes\\
%100, rue des maths 38610 Gières  - France}
%\end{center}
%\smallskip
%Le problème de Lehmer et le théorème de Dobrowolski se généralisent en dimension supérieure.
%%On dit que $Y$ est une variété de torsion si $Y$ est une réunion finie de translatés de sous--tores de $\Gm^n$ par des points de torsion.
%Rappelons que $\muess(Y)=0$ si et seulement si $Y$ est une variété de torsion.
%Le problème consistant à déterminer une bonne minoration pour le minimum essentiel d'une variété
%qui n'est pas de torsion a donné lieu à plusieurs travaux.
%
Dans le but de déterminer une bonne minoration pour le minimum essentiel d'une variété qui n'est pas de torsion, F. Amoroso et S. David (\cite[Définition 1.1, p.~337]{amoroso2003minoration}) ont introduit la notion d'indice d'obstruction. Étant donnés  deux sous--variétés irréductibles $Y\subsetneq W$ de $\Gm^n$ et un corps $\K\subseteq\Qb$, on définit l'indice d'obstruction $\omega_\K(Y,W)$ de $Y$ relatif à $W$ comme le minimum de
\begin{equation}
\label{eq:obstr}
\left(\frac{\deg(Z)}{\deg(W)}\right)^{1/\codim_W(Z)}
\end{equation}
où $Z$ parcourt les sous--variétés strictes de $W$, définies sur $\K$ et contenant $Y$, et où l'on a noté $\codim_W(Z) := \dim(W)-\dim(Z)$. Lorsque $\K=\Qb$, on omet l'indice $\Qb$ et lorsque $W=\Gm^n$, on notera simplement $\omega_\K(Y)$ l'indice d'obstruction $\omega_\K(Y,\Gm^n)$. 

Nous donnerons ici une minoration {\sl à la Dobrowolski} de l'indice d'obstruction. Cette minoration se déduit du résultat principal de~\cite{amoroso2012} à 
l'aide du lemme de transfert ci--dessous. Il s'agit d'une version plus précise du cas torique du théorème 3.7 de~\cite{remond2017generalisations}. Plus généralement, ce lemme permet de déduire des minorations du minimum essentiel en fonction de l'indice d'obstruction relatif d'une variété qui n'est pas de torsion à partir de minorations du minimum essentiel d'une variété tranverse ou faiblement transverse (\ie{} qui n'est contenue dans aucun translaté de sous--tore ou sous--variété de torsion). 

\begin{lemme}[G. Rémond]
\label{lemme:remond}
Soient $T\subseteq\Gm^n\hookrightarrow \P_n$ un sous--tore de dimension~$t$ et $Y\subsetneq T$ une sous--variété irréductible de $\Gm^n$ de dimension $y$. Il existe alors une sous--variété irréductible $W$ de $\Gm^t$ telle que~: 
\begin{enumerate}[1)]
\item $\dim(W)=y$~;
\item $\muess(W)\leq \binom{n}{t}\deg(T)^{-1}\muess(Y)$~;
\item $\omega_\K(W)\leq \deg(T)\omega_\K(Y,T)$ pour tout corps $\K\subseteq\Qb$~;
\item $\psi(W)=Y$ pour une certaine isogénie $\psi\colon\Gm^{t}\rightarrow T$.
\end{enumerate}
\end{lemme}
\begin{proof}
On note $\Delta$ le maximum des déterminants $t\times t$ d'une base du réseau associé à $T$. Quitte à permuter les coordonnées, on peut supposer qu'un des déterminants qui réalisent le maximum $\Delta$ est celui correspondant aux $t$ premières coordonnées de cette base. Par le lemme~\ref{lem:geom-nombres} 1) on a alors~:
\begin{equation}
\label{app}
\Delta\leq\deg(T)\leq\binom{n}{t}\Delta.
\end{equation} 
Considérons le morphisme d'oubli de coordonnées $\iota\colon\Gm^n\rightarrow \Gm^{t}$ qui envoie $(x_1,\dots,x_n)$ dans $(x_1,\dots,x_{t})$ et l'isogénie $\varphi\colon T\rightarrow\Gm^{t}$ composée de $\iota$ avec l'inclusion $T\hookrightarrow\Gm^n$. Par construction, $\deg(\varphi)=\Delta$. Soit $\psi\colon\Gm^{t}\rightarrow T$ l'isogénie duale. On a donc le diagramme suivant, où $[l]\colon\Gm^n\to\Gm^n$ pour $l\in\Z$ dénote le morphisme $\xg\mapsto\xg^l=(x_1^l ,\ldots,x_n^l)$, 
$$
\begin{tikzpicture}[node distance = 1.5cm, auto]
\node (Y) {$Y$};
\node (T) [right of=Y, node distance = 1.5cm] {$T$};
\node (Gmn) [right of=T, node distance = 1.5cm] {$\Gm^n$};
\node (GmT1) [below of=Y] {$\Gm^{t}$};
\node (GmT2) [below of=Gmn] {$\Gm^{t}$};
\draw[right hook->] (Y) to node {} (T);
\draw[right hook->] (T) to node {} (Gmn);
\draw[->] (GmT1) to node {\footnotesize $[\Delta]$} (GmT2);
\draw[->] (GmT1) to node {\footnotesize $\psi$} (T);
\draw[->] (T) to node {\footnotesize $\varphi$} (GmT2);
\draw[->] (Gmn) to node {\footnotesize $\iota$ oubli de coordonnées} (GmT2);
\end{tikzpicture}
$$
%où pour $[\Delta]$ désigne le morphisme de \og multiplication\fg{} par $\Delta$ dans une puissance de $\Gm$. 
On choisit pour $W$ une composante irréductible de $\psi^{-1}(Y)$. Les affirmations 1) et 4) sont alors claires par construction.

Pour montrer 2), on remarque que $[\Delta](W)=(\varphi\circ\psi)(W)=\varphi(Y)=\iota(Y)$ et donc $\Delta\muess(W)=\muess(\iota(Y))\leq\muess(Y)$. On obtient l'inégalité souhaitée en minorant $\Delta$ par~\eqref{app}.

Montrons 3). Soient $\K\subseteq\Qb$ un corps et $Z$ un fermé de Zariski défini sur $\K$ qui satisfait $Y\subseteq Z\subsetneq T$ et qui réalise le minimum dans la définition de $\omega_\K(Y,T)$~:
\begin{equation}
\label{eq:omega}
\left(\frac{\deg(Z)}{\deg(T)}\right)^{1/\codim_T(Z)}=\omega_\K(Y,T).
\end{equation}
On majore maintenant le degré de $\psi^{-1}(Z)$. Soit $E\subseteq\P^{t}$ un espace linéaire. Alors~:
$$
\psi^{-1}(\varphi^{-1}(E)\cap Z)=(\varphi\circ\psi)^{-1}(E)\cap\psi^{-1}(Z)=[\Delta]^{-1}(E)\cap\psi^{-1}(Z).
$$
Or $\deg([\Delta]^{-1}(E))=\Delta^{\codim_T(E)}$, donc si $E$ est suffisamment générique et de dimension $\codim_T(Z)$,
\begin{equation*}
\#\big([\Delta]^{-1}(E)\cap\psi^{-1}(Z)\big)=\Delta^{\dim(Z)}\deg\psi^{-1}(Z)
\end{equation*}
et
$$
\#\big(\psi^{-1}(\varphi^{-1}(E)\cap Z)\big)\leq \deg(\psi)\,\#(\varphi^{-1}(E)\cap Z)\leq \deg(\psi)\deg(Z),
$$
car $\varphi$ provient du morphisme $\iota$ d'oubli de coordonnées. On déduit alors des trois dernières formules centrées~:
$$
\Delta^{\dim(Z)} \deg\big(\psi^{-1}(Z)\big)\leq \deg(\psi)\deg(Z).
$$
D'où, en sachant que $\omega_\K(W)\leq \deg\big(\psi^{-1}(Z)\big)^{1/\codim\psi^{-1}(Z)}$, $\codim\psi^{-1}(Z)=\codim_T(Z)=d-\dim(Z)$, 
$\deg(\psi)=\deg([\Delta])/\deg(\varphi)=\Delta^{d-1}$, $\Delta\leq\deg(T)$ (\cf~\eqref{app}) et $\deg(T)^{-1}\deg(Z)=\omega_\K(Y,T)^{\codim_T(Z)}$ (\cf~\eqref{eq:omega})
\begin{align*}
\omega_\K(W)
&\leq \big(\Delta^{d-1-\dim(Z)}\deg(Z)\big)^{1/\codim\psi^{-1}(Z)}\\
&= \big(\Delta^{\codim_T(Z)-1}\deg(Z)\big)^{1/\codim_T(Z)}\\
&\leq \big(\deg(T)^{\codim_T(Z)-1}\deg(Z)\big)^{1/\codim_T(Z)}\\
&=\deg(T)\big(\deg(T)^{-1}\deg(Z)\big)^{1/\codim_T(Z)}\\
&=\deg(T)\omega_\K(Y,T).
\end{align*}
\end{proof}

À partir de ce lemme de transfert, on déduit une version fonctorielle du théorème de Dobrowolki en dimension supérieure.
\begin{corollaire}
\label{cor:muess}
Soit $Y$ une sous--variété irréductible de $\Gm^n$ qui n'est pas de torsion et soit $T\subseteq \Gm^n$ la plus petite sous--variété de torsion\footnote{On rappelle que l'intersection de deux variétés de torsion est soit non vide, soit réunion de variétés de torsion.} contenant $Y$. On suppose que $T$ est un tore et on note $t:=\dim(T)$ et $y:=\dim(Y)$. Alors~: 
$$ 
\muess(Y) \geq \tbinom{n}{t}^{-1}\omega_\Q(Y,T)^{-1}\big(935 t^5\log(t^2\deg(T)\omega_\Q(Y,T))\big)^{-(t-y)(t-y +1)(t+1)}.
$$
\end{corollaire}
\begin{proof}
Par le lemme~\ref{lemme:remond}, il existe une sous--variété irréductible $W$ de $\Gm^t$ ayant les propriétés de 1) à 4). Cette sous--variété est faiblement transverse par l'hypothèse faite sur $T$. On peut donc lui appliquer le corollaire 1.5 de~\cite{amoroso2012}~:
$$ 
\muess(W) \geq \omega_\Q(W)^{-1}\big(935 t^5\log(t^2\omega_\Q(W))\big)^{-(t-y)(t-y+1)(t+1)}.
$$
Par les points 2) et 3) (avec $\K=\Q$) du lemme~\ref{lemme:remond} on a~:
$$ 
\muess(W) \leq \tbinom{n}{t}\deg(T)^{-1}\muess(Y)\quad\text{ et }\quad \omega_\Q(W)\leq \deg(T)\omega_\Q(Y,T).
$$
La minoration annoncée s'ensuit. 
\end{proof}

\section{Preuve du théorème principal}
\label{section:proof}
Dans ce paragraphe, on prouve le théorème~\ref{theo:main}. Remarquons que l'énoncé est vide si $n=1$. On supposera donc dans la suite $n\geq2$. 

Soit $V$ la sous--variété de $\Gm^n$ définie par $F_1=\dots=F_s=0$. En suivant~\cite{schmidt1996heights}, on note $V^u$ la réunion des sous--variétés de torsion contenues dans $V$ et $V^*:=V \setminus V^u$.

Supposons d'abord $\alphag:=\xi^\ag:=(\xi^{a_1},\dots,\xi^{a_n})\in V^u$. Il existe alors un sous--tore $T$ et un point de torsion $\zetag$ tels que $\alphag \in \zetag T\subseteq V$. Soit $H$ un sous--groupe algébrique maximal parmi les sous--groupes $H$ tels que $T\subseteq H\subseteq \alphag^{-1}V$. D'après le lemme~\ref{lemme:schmidt}, il existe un sous--groupe $\Lambda$ de $\Z^n$ engendré par des vecteurs dans $\D ( \cup_{j=1}^s \Supp(F_i))$ tel que $H=H_\Lambda$. Soit $\bg\in \D ( \cup_{j=1}^s \Supp(F_j))$.
Donc $1=(\zetag^{-1}\alphag)^\bg=\zetag^{-\bg}\xi^{\langle\ag,\bg\rangle}$, d'où $\bg\in\ag^\perp$ car $\zetag$ est de torsion et $\xi$ n'est pas une racine de l'unité. Par ailleurs, $\Vert \bg\Vert_2\leq 2 \max_j\deg(F_j)$, puisqu'on peut supposer que chaque $F_j$ n'a que des exposants positifs et pas de mon\^omes en facteur, et donc la majoration souhaitée~\eqref{eq:maj_b1} pour $\Vert \bg\Vert_2$ est {\it a fortiori} satisfaite.

On remarquera que pour l'instant nous n'avons pas utilisé l'hypothèse qu'il existe une composante irréductible $W$ de codimension $\geq 2$ de $V$ qui passe par $\xi^\ag$.  On le fait maintenant pour montrer, sous l'hypothèse additionnelle $\alphag\in V^*$, le résultat plus précis suivant. Le complément à la fin de la proposition n'est pas nécessaire pour la preuve du théorème~\ref{theo:main}, mais il sera utilisé pour les applications à l'algorithmique des polynômes lacunaires.

\begin{proposition}
\label{prop:main}
Soient $F_1,\dots,F_s\in\Z[x_1^{\pm1},\ldots,x_n^{\pm1}]$. Notons 
$$
D:=\max_{1\leq i\leq s}\deg(F_i), \quad\hbox{et}\quad  h_1:=\max_{1\leq i\leq s}\log\Vert F_i\Vert_{1},
$$
et $\tilde{h}_1:=h_1+ (n+13)\log (n+1)D$. Soient également $\ag\in\Z^n\setminus\{\mathbf{0}\}$ et $\xi\in \Qb^*$ différent d'une racine de l'unité. Soit $V$ la sous--variété de $\Gm^n$ définie par $F_1=\dots=F_s=0$ et notons $\alphag:=\xi^\ag=(\xi^{a_1},\dots,\xi^{a_n})\in V^*$. On suppose 
$$
\alphag\in V^*\cap W,
$$ 
où $W$ est une composante irréductible de codimension $k\geq 2$ de $V$. 
Alors on a:
\begin{equation}
\label{eq:minimum}
\lambda_{k-1}(\ag^\perp) \leq (4n)^{17n}\,\tilde{h}_1 D^{k-1}\log(\tilde{h}_1D^k)^{2n}.
\end{equation}
En particulier, il existe un vecteur non nul $\bg \in \ag^\perp$ qui satisfait $\Vert \bg \Vert_2 \leq (4n)^{17n}\,\tilde{h}_1 D^{k-1}\log(\tilde{h}_1D^k)^{2n}$.

\noindent{\bf Complément.}
Soit $\{\bg_1,\dots,\bg_{n-1}\}$ une base réduite de $\ag^\perp$  au sens de l'algorithme~\text{LLL} . Pour $i=0,\ldots,n-1$, on note $\Lambda_i$ le sous--groupe de $\Z^n$ engendré par les vecteurs $\bg_1,\dots,\bg_i$ et on pose $T^{(i)} := H_{\Lambda_i}$.  Alors il existe\footnote{L'entier $m$ pourra dépendre de la base.} $m\in\{k-2,\dots,n-2\}$ tel que 
pour tout $i\in\{m+1,\dots,n-1\}$, toute composante irréductible de $W\cap T^{(i)}$ contenant $\alphag$ est de dimension $n-i-1$ et\footnote{\cf{} \eqref{eq:minima} pour la définition des minima successifs.}
\begin{equation}
\label{eq:minimum2}
\lambda_{m+1}(\ag^\perp) \leq 2^{n^2/2} (4n)^{17n} \,\tilde{h}_1 D^{k-1}\log(\tilde{h}_1D^k)^{2n}.
\end{equation}
\end{proposition}

\begin{proof}
Nous aurons besoin à plusieurs reprises de majorations d'analyse élémentaire. %Nous n'en expliciterons pas les calculs afin de ne pas alourdir l'exposition~; ils sont par ailleurs facilement vérifiables par un logiciel de calcul formel. 
Pour la commodité du lecteur nous avons signalé ces inégalités par les symboles \og $\leq^{(*)}$\fg{} et \og $\geq^{(*)}$\fg{}. Elles seront prouvées dans l'annexe.\\

Soit $\{\bg_1,\dots,\bg_{n-1}\}$ une base de $\ag^\perp$. Pour $i=0,\ldots,n-1$, on note $\Lambda_i$ le sous--groupe de $\Z^n$ engendré par les vecteurs $\bg_1,\dots,\bg_i$ et on pose $T^{(i)} := H_{\Lambda_i}$. On remarque que les $T^{(i)}$ sont des tores car les $\Lambda_i$ sont saturés. On a ainsi la suite de sous--tores suivante~:
$$ 
H_{\ag^\perp}=T^{(n-1)} \subseteq T^{(n-2)} \subseteq \dots \subseteq T^{(1)} \subseteq T^{(0)} = \Gm^n,
$$
de codimension $\codim(T^{(i)})=i$. Remarquons que pour $\bg\in\ag^\perp$ on a $\alphag^\bg=(\xi^\ag)^\bg=1$. Pour $i=0,\ldots,n-1$ on a donc $\alphag\in H_{\ag^\perp}= T^{(n-1)}\subseteq T^{(i)}$ et $\alphag\in W\cap T^{(i)}$. 

On a $\dim(W\cap  T^{(i)})\leq\dim( T^{(i)})\leq n-i$. Si $\dim(W\cap  T^{(i)})= n-i$ alors $ T^{(i)}\subseteq V$ et ceci contredit le fait que $\alphag \in V^*$. Donc
\begin{equation}
\label{eq:D_i}
\dim(W\cap  T^{(i)})\leq n-i-1.
\end{equation}

Pour $i\in\{0,\dots,n-1\}$, on note~:
$$
d_i:= \min_Y\dim Y,
$$
où $Y$ parcourt l'ensemble des composantes irréductibles de $W\cap  T^{(i)}$ contenant $\alphag$ (cet ensemble étant non vide car $\alphag\in W\cap T^{(i)}$). Par ce qui précède, on a~:
$$
0\leq d_i\leq n-i-1.
$$
Par hypothèse, $d_0=\dim (W) = n-k\leq n-2$.  On peut alors définir $m$ comme le plus grand entier vérifiant~:
$$ 
0\leq m\leq n-1,\qquad d_m\leq  n-m-2. 
$$
Remarquons que $m\leq n-2$, car $d_{n-1}>-1=n-(n-1)-2$. Donc $m+1\leq n-1$ et (par maximalité de $m$) $d_i = n-i-1$ pour $i\geq m+1$. 
Donc pour $i\geq m+1$, toute composante irréductible de $W \cap T^{(i)}$ contenant $\alphag$ est de dimension $n-i-1$. Cela montre la première affirmation du complément à l'énoncé de la proposition~\ref{prop:main}.
% (on remarquera que pour cela nous n'avons pas utilisé le fait que la base $\{\bg_1,\dots,\bg_{n-1}\}$ est réduite).}

Nous allons maintenant prouver la majoration~\eqref{eq:minimum}. Par ce qui précède, on a en particulier~:
\begin{scolie}
\label{scolie:W}
Il existe une composante irréductible $Y$, commune à $W \cap  T^{(m)}$ et $W \cap  T^{(m+1)}$, contenant $\alphag$, de dimension $n-m-2$.
\end{scolie}
On fixe maintenant une sous--variété irréductible $Y$ de $\Gm^n$ qui satisfait les affirmations de la scolie~\ref{scolie:W}. Comme $\dim(Y) \leq \dim(W)$, on a $n-m-2\leq n-k$ et donc 
\begin{equation}
\label{eq:m}
m+1\geq k-1\geq 1.
\end{equation}

On note $\Q(Y)$ le corps de définition de $Y$ sur $\Q$ et
$$
Y':=\bigcup_{\sigma} \sigma(Y)\quad\hbox{et}\quad W':=\bigcup_{\sigma} \sigma(W)
$$
où $\sigma$ parcourt $\Gal(\Q(Y)/\Q)$. On a~:
\begin{equation}
\label{eq:rational}
Y'\subseteq W'\cap  T^{(m)},
\end{equation}
car $ T^{(m)}$ est définie sur $\Q$. 

Nous avons besoin de deux lemmes auxiliaires. Le premier donne une majoration du degré et du minimum essentiel de $Y'$ à l'aide des résultats des paragraphes~ \ref{section:hauteurs} et \ref{section:bezout}. Rappelons que $\tilde{h}_1=h_1+ (n+13)\log (n+1)D$.
\begin{lemme*}
\begin{align}
\label{eq:degY'}
\deg(Y')&\leq D^k\deg(T^{(m)}),\\
\label{eq:muess}
\deg(Y')\muess(Y)&\leq n\tilde{h}_1 D^{k-1}\deg(T^{(m)}).
\end{align}
\end{lemme*}
\begin{proof}
Par le scolie~\ref{scolie:W} et par définition de $W'$ et $Y'$, cette dernière est une composante $\Q$--irréductible de $W' \cap  T^{(m)}$. D'après le corollaire~\ref{coro:bezoutarithm}, $\deg(W')\leq D^k$. La majoration~\eqref{eq:degY'} suit alors du théorème de Bézout~\ref{theo:bezoutarithm} (en choisissant $\K=\Q$).

La majoration de $\h(Y')$ demande plus de calculs. Par l'inégalité de Zhang (\ref{eq:Zhang}) entre minimum essentiel et hauteur normalisée et par la relation~\eqref{eq:hP-hGmn} entre hauteur normalisée et hauteur projective,
$$
\deg(Y)\muess(Y)\leq\h(Y)\leq\hP(Y)+ \tfrac{7}{2}(n+1)\log(n+1)\deg(Y).
$$
Donc, par linérarité et Galois--invariance du degré et de la hauteur projective et par~\eqref{eq:degY'},
\begin{multline}
\label{eq:muess1}
\deg(Y')\muess(Y)\leq\hP(Y')+ \tfrac{7}{2}(n+1)\log(n+1)\deg(Y')\\
\leq\hP(Y')+ \tfrac{7}{2}(n+1)\log(n+1)D^k\deg(T^{(m)}).
\end{multline}
Pour majorer $\hP(Y')$ on utilise~:\\[0.4cm]
$-$ Le théorème de Bézout~\ref{theo:bezoutarithm}
$$
\hP(Y')\leq\hP(W')\deg(T^{(m)}) + \hP(T^{(m)})\deg(W')+c_1(n)\deg(W')\deg(T^{(m)}),
$$
où $c_1(n):=n\log2+\tfrac12\sum_{i=0}^n\sum_{j=0}^n\frac1{i+j+1}$.\\
$-$ Les majorations du degré et de la hauteur projective de $W'$ (\cf{} corollaire~\ref{coro:bezoutarithm}),
$$
\deg(W') \leq D^k\quad\hbox{et}\quad\hP(W') \leq n h_1D^{k-1}+c_2(n)D^k,
$$
où $c_2(n):=n\sum_{i=1}^{n}\sum_{j=1}^i\frac{1}{j}$.\\
$-$ La relation~\eqref{eq:hP-hGmn} entre hauteur normalisée (qui est nulle) et hauteur projective du tore $T^{(m)}$, 
$$
\hP(T^{(m)})\leq \tfrac{7}{2} (n+1)\log(n+1)\deg(T^{(m)}).
$$~\\[0.1cm]
Par les quatre inégalités ci--dessus, 
\begin{multline*}
\hP(Y')
\leq \left(\hP(W')+\left(\tfrac{\hP(T^{(m)})}{\deg(T^{(m)})}+c_1(n)\right)\deg(W')\right)\deg(T^{(m)})\\
\leq\left(n h_1+\left(\tfrac{7}{2} (n+1)\log(n+1)+c_1(n)+c_2(n)\right)D\right)D^{k-1}\deg(T^{(m)}).
\end{multline*}
En majorant $\hP(Y')$ dans~\eqref{eq:muess1} à l'aide de cette dernière inégalité, on trouve~:
$$
\deg(Y')\muess(Y)\leq n(h_1+c_3(n)D)D^{k-1}\deg(T^{(m)}),
$$
où 
\begin{align*}
c_3(n)
&:=7(1+1/n)\log(n+1)+c_1(n)/n+c_2(n)/n\\
&=7 (1+1/n)\log(n+1)+\log(2)+\frac1{2n}\sum_{i=0}^n\sum_{j=0}^n\frac1{i+j+1}+\sum_{i=1}^{n}\sum_{j=1}^i\frac{1}{j}\\
&\leq^{\eqref{inega:1}} (n+13)\log (n+1). % inega:1
\end{align*}
\end{proof}

Le point crucial de la preuve est la majoration suivante de $\deg(T^{(m+1)})$ en fonction de $\deg(T^{(m)})$, majoration qui repose sur le corollaire~\ref{cor:muess}.

\begin{lemme*}
On a~:
\begin{multline}
\label{eq:tores1}
\deg(T^{(m+1)})\\ \leq n2^n\tilde{h}_1 D^{k-1}\deg(T^{(m)})\left(935n^5\log\left(n^2D^n\deg(T^{(m)})\right)\right)^{2n}.
\end{multline}
\end{lemme*}

\begin{proof}
Soit $T\subseteq \Gm^n$ la plus petite sous--variété de torsion contenant $Y$. On a $Y\subseteq T^{(m+1)}$ et donc $Y\subseteq T \subseteq  T^{(m+1)}$. Par ailleurs $\dim(Y) = n-m-2$ et $\dim(T^{(m+1)})=n-m-1$ et donc, par irréductibilité, on a soit $T=Y$ soit $T=T^{(m+1)}$. Or $Y$ n'est pas de torsion car $\alphag\in Y\subseteq V$ et $\alphag \in V^*$. On en déduit que $T= T^{(m+1)}$. Comme $\codim_{T^{(m+1)}}(Y)=1$ on a (\cf{} définition~\ref{eq:obstr})~:
$$
\omega_\Q(Y,T^{(m+1)}) \leq \frac{\deg(Y')}{\deg(T^{(m+1)})}.
$$
D'après le corollaire~\ref{cor:muess}, avec $t=\dim(T)= n-m-1$ et $y=\dim(Y)=n-m-2$, on a~:
\begin{align*}
\muess(Y) 
&\geq \tbinom{n}{n-m-1}^{-1}\omega_\Q(Y,T^{(m+1)})^{-1}\\
&\;\times\big(935(n-m-1)^5\log((n-m-1)^2\deg(T^{(m+1)})\omega_\Q(Y,T))\big)^{-2(n-m)}\\
&\geq 2^{-n}\omega_\Q(Y,T^{(m+1)})^{-1}\left(935n^5\log\left(n^2\deg(T^{(m+1)})\omega_\Q(Y,T)\right)\right)^{-2n}.
\end{align*}
Donc
$$
\deg(T^{(m+1)}) \leq 2^n\deg(Y')\muess(V)\left(935n^5\log(n^2\deg(Y'))\right)^{2n}. 
$$
En majorant $\deg(Y')\muess(V)$ par~\eqref{eq:muess} et en majorant $\log(n^2\deg(Y'))$ par~\eqref{eq:degY'} (et $D^k$ par $D^n$), on obtient la majoration souhaitée~\eqref{eq:tores1}.
\end{proof}
Pour $i=1,\ldots,n-1$ notons pour abréger $\lambda_i:=\lambda_i(\ag^\perp)$ le $i$--ème minimum successif (\cf{} \eqref{eq:minima}) du réseau $\ag^\perp$. On va déduire à partir de ce lemme une majoration de $\lambda_{m+1}$. On garde les notations ci--dessus et on choisit pour $\{\bg_1,\dots,\bg_{n-1}\}$ une base de $\ag^\perp$ qui satisfait l'inégalité du lemme~\ref{lem:minkowski} 2). D'après les points 1) et 2) du lemme~\ref{lem:minkowski}, on encadre le volume des tores $T^{(i)}$ par des produits de minima successifs~:
\begin{equation}
\label{eq:degTore}
2^{-i}V_i \binom{n}{i}^{-1/2}\prod_{j=1}^i \lambda_j\leq\deg(T^{(i)})\leq\binom{n}{i}^{1/2} \prod_{j=1}^i \lambda_j
\end{equation}
où $V_i=\pi^{i/2}\Gamma(1+i/2)^{-1}$ (cet encadrement est encore valable pour $i=0$, car $\deg(T^{(0)})=\deg(\Gm^n)=1$). On minore $\deg(T^{(m+1)})$ et on majore $\deg(T^{(m)})$ par~\eqref{eq:degTore}. On obtient~:
\begin{equation}
\label{eq:tores2}
\frac{\deg(T^{(m+1)})}{\deg(T^{(m)})}\geq c(n,m)^{-1}\lambda_{m+1}
\end{equation}
où $c(n,m):=2^{m+1}\tbinom{n}{m+1}^{1/2}\tbinom{n}{m}^{1/2} V_{m+1}^{-1}$. En majorant à nouveau $\deg(T^{(m)})$ par~\eqref{eq:degTore} et en tenant compte de l'inégalité $\lambda_{j} \leq \lambda_{m+1}$ pour $j=1,\ldots,m+1$,
\begin{equation}
\label{eq:tores3}
\deg(T^{(m)})\leq \tbinom{n}{m}^{1/2} \lambda_{m+1}^m
\leq 2^{n/2} \lambda_{m+1}^n.
\end{equation}
Par~\eqref{eq:tores1},~\eqref{eq:tores2} et~\eqref{eq:tores3} on a~:
$$
c(n,m)^{-1}\lambda_{m+1}\leq n2^n\tilde{h}_1 D^{k-1}
\left(935n^5\log\left(n^22^{n/2}D^n\lambda_{m+1}^n\right)\right)^{2n}.
$$
En tenant compte des inégalités 
$$
n2^n\!\!\max\limits_{0\leq m \leq n-2}c(n,m)\leq^{\eqref{inega:2}} n^{2n} % inega:2
\quad\hbox{et}\quad n^2 2^{n/2}\leq^{\eqref{inega:3}} 3^n, % inega:3
$$
on en déduit~:
$$
\lambda_{m+1}\leq \tilde{h}_1 D^{k-1}\left(935n^7\log(3D\lambda_{m+1})\right)^{2n}.
$$
On pose $x_0:=\lambda_{m+1}^{1/2n}$; la dernière inégalité se lit alors~:
\begin{equation}
\label{eq:x0}
x_0 \leq(\tilde{h}_1 D^{k-1})^{1/2n}\,935n^7\log\left(3Dx_0^{2n}\right)=\alpha\log(\beta x_0)
\end{equation}
avec $\alpha:=2\cdot935n^8(\tilde{h}_1 D^{k-1})^{1/2n}$ et $\beta:=(3D)^{1/(2n)}$. 

\begin{lemme*}
\begin{equation}
\label{eq:l1}
x_0\leq 2\alpha\log(\alpha\beta).
\end{equation}
\end{lemme*}
\begin{proof} 
Notons que $x_0>1/\beta$, car sinon $x_0\leq \alpha\log(\beta x_0)\leq0$. Posons $f(x) := x/\log (\beta x)$ pour $x> 1/\beta$. Cette fonction est décroissante sur $]1/\beta, e/\beta[$
et croissante sur $[e/\beta,+\infty[$. Elle admet un minimum en $e/\beta$ qui vaut $e/\beta$. 

L'inéquation $x\leq \alpha\log(\beta x)$ est équivalente à $f(x)\leq \alpha$. On a donc $\alpha\geq e/\beta$, d'où $x_1:= e^{1/2} \alpha \log (\alpha\beta)\geq e/\beta$. Si $x_0< e/\beta$ on a alors $x_0\leq x_1$. Supposons donc $x_0\geq e/\beta$. On a~:
$$
\frac{f(x_1)}{\alpha}
= \frac{e^{1/2}\log (\alpha\beta)}{1/2+\log (\alpha\beta) + \log \log (\alpha\beta))}
\geq 1
$$
(car $(e^{1/2}-1)t - \log(t) - 1/2\geq^{\eqref{inega:4}} 0$ pour $t\geq 1$) et donc $\alpha \leq f(x_1)$. % inega:4
Par~\eqref{eq:x0}, on a aussi $f(x_0)\leq \alpha$. Donc, par croissance de $f$ et en majorant $\sqrt{e}$ par $2$, on en déduit que $x_0\leq x_1\leq 2\alpha \log (\alpha\beta)$.
\end{proof}
D'après~\eqref{eq:l1} on obtient 
\begin{align*}
\lambda_{m+1}=x_0^{2n}
&\leq \left(2\alpha\log(\alpha\beta)\right)^{2n}\\
&= \tilde{h}_1 D^{k-1}\left(4\cdot 935n^8\log(\alpha\beta)\right)^{2n}\\
&= \tilde{h}_1 D^{k-1}\left(1870n^7\log\left((\alpha\beta)^{2n}\right)\right)^{2n}\\
&= \tilde{h}_1 D^{k-1}\left(1870n^7\log\left(3(1870n^8)^{2n}\tilde{h}_1D^k\right)\right)^{2n},
\end{align*}
et, en tenant compte de l'inégalité $x+y\leq xy$, valable pour $x,y\geq2$ et que l'on peut donc utiliser ici car
$\log(\tilde{h}_1D^k)\geq \log((n+13)\log (n+1))\geq^{\eqref{inega:5}} 2$, % inega:5
\begin{align*}
\lambda_{m+1}
&\leq\tilde{h}_1 D^{k-1}\left(1870n^7\log\left(3(1870n^8)^{2n}\right)\right)^{2n}\log(\tilde{h}_1D^k)^{2n}\\
&\leq\left(2\cdot1870n^8\log(3\cdot1870n^8)\right)^{2n}\tilde{h}_1 D^{k-1}\log(\tilde{h}_1D^k)^{2n}\\
&\leq(4n)^{17n}\,\tilde{h}_1 D^{k-1}\log(\tilde{h}_1D^k)^{2n},
\end{align*}
où l'on a utilisé l'inégalité $2\cdot1870n^8\log(3\cdot 1870n^8)\leq^{\eqref{inega:6}}(4n)^{17/2}$. % inega:6
En tenant compte que $\lambda_{k-1}\leq\lambda_{m+1}$ par~\eqref{eq:m}, on en déduit \eqref{eq:minimum}, ce qui achève la preuve de la proposition~\ref{prop:main}.\\ 

Nous allons maintenant terminer la preuve du complément. On a déjà montré que pour tout $i\in\{m+1,\dots,n-1\}$, toute composante irréductible de $W\cap T^{(i)}$ contenant $\alphag$ est de dimension $n-i-1$. Pour terminer la preuve il nous reste à montrer la deuxième affirmation du complément. On choisit pour $\{\bg_1,\dots,\bg_{n-1}\}$ une base de $\ag^\perp$ qui soit réduite  au sens de l'algorithme \text{LLL}. Elle satisfait donc l'inégalité du lemme~\ref{lem:minkowski} 3). D'après les points 1) et 3) de ce lemme,
\begin{equation}
\label{eq:degTorebis}
2^{-i}V_i \binom{n}{i}^{-1/2}\prod_{j=1}^i \lambda_j\leq\deg(T^{(i)})\leq 2^{i(n-2)/2}\binom{n}{i}^{1/2} \prod_{j=1}^i \lambda_j.
\end{equation}
En procédant comme auparavant, on obtient~:
\begin{align*}
\frac{\deg(T^{(m+1)})}{\deg(T^{(m)})}
&\geq c'(n,m)^{-1}\lambda_{m+1},\\
n^2 \deg(T^{(m)})&%\leq 2^{m(n-2)}\tbinom{n}{m}^{1/2} \lambda_{m+1}^m
\leq n^2 2^{m(n-2)/2+n/2} \lambda_{m+1}^n\leq 2^{n^2} \lambda_{m+1}^n,
\end{align*}
où cette fois $c'(n,m):=2^{m(n-2)/2+m+1}\tbinom{n}{m+1}^{1/2}\tbinom{n}{m}^{1/2} V_{m+1}^{-1}$ et où l'on a utilisé l'inégalité $n^22^{m(n-2)/2+n/2}\leq n^22^{(n-2)^2+n/2}\leq^{\eqref{inega:7}} 2^{n^2}$. Donc, par~\eqref{eq:tores1},
\begin{align*}
c'(n,m)^{-1}\lambda_{m+1}
%&\leq n2^n\tilde{h}_1 D^{k-1}\left(935n^5\log\left(n^22^{m(n-2)+n/2}D^n\lambda_{m+1}^n\right)\right)^{2n}\\
&\leq n2^n\tilde{h}_1 D^{k-1}\left(935n^6\log(2^nD\lambda_{m+1})\right)^{2n}.
\end{align*}
En tenant compte de~:
$$
n2^n(935n^6)^{2n}\!\!\max\limits_{0\leq m\leq n-2}c'(n,m)\leq^{\eqref{inega:8}} 2^{n^2/2}(2^{12}n^{25/4})^{2n}
$$
on en déduit~:
$$
\lambda_{m+1}\leq \tilde{h}_1 D^{k-1}2^{n^2/2}\left(2^{12}n^{25/4}\log(2^nD\lambda_{m+1})\right)^{2n}.
$$
On pose $x_0:=\lambda_{m+1}^{1/2n}$; la dernière inégalité se lit alors~:
$$
x_0 \leq(\tilde{h}_1 D^{k-1})^{1/2n}\cdot2^{n/4}\cdot 2^{12}n^{25/4}\log(2^nDx_0^{2n})=\alpha\log(\beta x_0)
$$
avec $\alpha:=(\tilde{h}_1 D^{k-1})^{1/2n}\cdot2^{n/4}\cdot 2^{12}n^{25/4}\cdot 2n$ et $\beta:=(2^nD)^{1/2n}$. D'après~\eqref{eq:l1} on obtient 
\begin{align*}
\lambda_{m+1}=x_0^{2n}
&\leq \left(2\alpha\log(\alpha\beta)\right)^{2n}\\
&= \tilde{h}_1 D^{k-1}\cdot2^{n^2/2}\left(2\cdot 2^{12}n^{25/4}\cdot 2n\cdot\log(\alpha\beta)\right)^{2n}\\
&= \tilde{h}_1 D^{k-1}\cdot2^{n^2/2}\left(2^{13}n^{25/4}\cdot\log\left((\alpha\beta)^{2n}\right)\right)^{2n}.
\end{align*}
En tenant compte de l'inégalité $x+y\leq xy$, valable pour $x,y\geq2$ (et que l'on peut donc utiliser par la majoration $\log(\tilde{h}_1D^k)\geq 2$ déjà mentionnée), on a~:
\begin{align*}
\log\left(\left(\alpha\beta)^{2n}\right)\right)
&=\log\left((2^{n/4}\cdot 2^{12}n^{25/4}\cdot 2n)^{2n}\,2^n\tilde{h}_1D^k\right)\\
&= n\log( 2^{27}\cdot 2^{n/2}\cdot n^{29/2})\log(\tilde{h}_1D^k)\\
&\leq^{\eqref{inega:9}} 2^4n^2\log(\tilde{h}_1D^k)
\end{align*}
Donc 
\begin{multline*}
\lambda_{m+1}
\leq 2^{n^2/2}(2^{17}n^{33/4})^{2n}\tilde{h}_1 D^{k-1} \log(\tilde{h}_1D^k)^{2n}\\
\leq^{\eqref{inega:10}} 2^{n^2/2} (4n)^{17n}  \tilde{h}_1 D^{k-1} \log(\tilde{h}_1D^k)^{2n}.
\end{multline*}
\end{proof}

\section*{Annexe. Preuves des inégalités numériques.}
\label{annexe:inegalite}

\begin{inegalite}
  \label{inega:1}
  Pour $n\geq 2$ on a~:
\begin{multline*}
7 (1+1/n)\log(n+1)+\log(2)+\frac1{2n}\sum_{i=0}^n\sum_{j=0}^n\frac1{i+j+1}+\sum_{i=1}^{n}\sum_{j=1}^i\frac{1}{j}\\[0.2cm]
\leq (n+13)\log (n+1)
\end{multline*}
\end{inegalite}
\begin{proof} Pour $n=2$ on vérifie par calcul. On suppose donc $n\geq 3$.
Pour un entier $k\geq 1$ notons $H_k:=\sum_{i=1}^k\tfrac1i$ le $k$-ième nombre harmonique. On a alors~:
$$
\frac1{2n}\sum_{i=0}^n\sum_{j=0}^n\frac1{i+j+1}
\leq\frac1{2n}\sum_{i=0}^n\sum_{j=0}^n\frac1{j+1}= \frac{n+1}{2n}H_{n+1}.
$$
et
$$
\sum_{i=1}^{n}\sum_{j=1}^i\frac{1}{j}=\sum_{j=1}^n\frac{n-j+1}{j}= (n+1) H_n-n.
$$
%où l'on a utilisé $(n+1)\log(1+1/n)\geq 1$. 
En utilisant l'inégalité $H_k\leq 1+\log k$, on peut donc majorer le membre de gauche de l'inégalité proposée par~: 
$$
7 (1+\tfrac1n)\log(n+1)+\log(2)+(n+\tfrac32+\tfrac1{2n})H_{n+1}-n\leq f(n)\\
$$
où pour $t>0$ on a noté
$$
f(t)=7 (1+\tfrac1t)\log(t+1)+\log(2)+(t+\tfrac32+\tfrac1{2t})(1+\log(t+1))-t.
$$
Soit $t_0=9.56...$ . La fonction $t\to f(t)/((t+13)\log(t+1))$ est décroissante dans l'intervalle $[3,t_0]$ et croissante pour $t>t_0$. De plus, $f(3)<1$ et $\lim_{t\to\infty}f(t)=1$. On en déduit la majoration souhaitée.
\end{proof}

\begin{inegalite}
\label{inega:2}
Pour $i\geq 1$ et $n\geq 2$ posons $V_i:=\pi^{i/2}\Gamma(1+i/2)^{-1}$ et 
$c(n,m):=2^{m+1}\tbinom{n}{m+1}^{1/2}\tbinom{n}{m}^{1/2} V_{m+1}^{-1}$. On a alors~:
$$
n2^n\!\!\max\limits_{0\leq m \leq n-2}c(n,m)\leq n^{2n}. 
$$
\end{inegalite}
\begin{proof}
Pour $n=2,3$, on la vérifie par calcul. Supposons donc $n\geq 4$. Pour $m\in\N$, $m\leq n-2$, on majore $V_{m+1}^{-1}\leq\Gamma(1+n/2)\leq(n/2)^{n/2}$ (car $\Gamma(1+x)<x^x$ pour $x>0$). On majore les deux binomiaux dans la définition de $c(n,m)$ par $2^n$. On a alors
$$
n 2^n c(n,m)\leq n 8^n (n/2)^{n/2}.
$$
Par ailleurs, $t\to f(t):=\log(t)+t\log(8)+(t/2)\log(t/2)-2t\log(t)$ est décroissante pour $t\geq 2$ et $f(4)=0$. Donc $n 8^n (n/2)^{n/2}\leq n^{2n}$.
\end{proof}

\begin{inegalite}
  \label{inega:3}
  Pour $n\geq 2$ on a $n^22^{n/2}\leq 3^n$.
\end{inegalite}
\begin{proof}
$t\to f(t):=2\log(t)+(t/2)\log(2)-t\log(3)$ est décroissante pour $t\geq 3$ et $f(2)$, $f(3)<0$. 
\end{proof}

\begin{inegalite}
  \label{inega:4}
  Pour $t\geq1$ on a $(e^{1/2}-1)t - \log(t) - 1/2\geq 0$.
\end{inegalite}
\begin{proof}
$t\to f(t):=(e^{1/2}-1)t - \log(t) - 1/2$ admet un minimum en $t_0=(e^{1/2}-1)^{-1}$ et $f(t_0)=\log(e^{1/2}-1)+1/2>0$.
\end{proof}

\begin{inegalite}
  \label{inega:5}
  Pour $n\geq 2$ on a $\log((n+13)\log (n+1))\geq 2$.
\end{inegalite}
\begin{proof}
On a~: $\log((n+13)\log (n+1))\geq \log(15\log (3))>2$.
\end{proof}

\begin{inegalite}
  \label{inega:6}
  Pour $n\geq 2$ on a $2\cdot1870n^8\log(3\cdot 1870n^8)\leq(4n)^{17/2}$.
\end{inegalite}
\begin{proof}
$t\to f(t):=\log(2\cdot1870 t^8)+\log\log(3\cdot 1870t^8)-\tfrac{17}2\log(4t)$ est décroissante pour $t\geq3$ et $f(2)$, $f(3)<0$.
\end{proof}

\begin{inegalite}
  \label{inega:7}
  Pour $n\geq 2$ on a $n^22^{(n-2)^2+n/2}\leq 2^{n^2}$.
\end{inegalite}
\begin{proof}
$t\to f(t):=2\log(t)+((t-2)^2+t/2)\log(2)-t^2\log(2)$ est décroissante pour $t\geq 2$ et $f(2)<0$. 
\end{proof}

\begin{inegalite}
\label{inega:8}
Pour $i\geq 1$ et $n\geq 2$ posons $V_i:=\pi^{i/2}\Gamma(1+i/2)^{-1}$ et 
$c'(n,m):=2^{m(n-2)+m+1}\tbinom{n}{m+1}^{1/2}\tbinom{n}{m}^{1/2} V_{m+1}^{-1}$. On a alors:
$$
n2^n(935n^6)^{2n}\!\!\max\limits_{0\leq m\leq n-2}c'(n,m)\leq 2^{n^2/2}(2^{12}n^{25/4})^{2n}.
$$
\end{inegalite}
\begin{proof}
Pour $m\in\N$, $m\leq n-2$, on majore $V_{m+1}^{-1}\leq(n/2)^{n/2}$ et les deux binomiaux dans la définition de $c'(n,m)$ par $2^n$ (\cf{} inégalité~\ref{inega:2}). On a alors:
$$
c'(n,m)\leq 2^{(n-2)^2/2+n-1}(n/2)^{n/2}
\leq 2^{n^2/2}(n/2)^{n/2}.
$$
Par ailleurs, 
\begin{multline*}
t\to f(t):=\log(t)+t\log(2)+2t(\log(935)+6\log(t))+(t/2)\log(t/2)\\-2t(12\log(2)+(25/4)\log(t))
\end{multline*}
est décroissante pour $t\geq 1$ et $f(2)<0$. Donc $n2^n(935n^6)^{2n}\cdot (n/2)^{n/2}\leq (2^{12}n^{25/4})^{2n}$.
\end{proof}

\begin{inegalite}
\label{inega:9}
Pour $n\geq 2$ on a: 
$$
\log(2^{27}\cdot 2^{n/2}\cdot n^{29/2})\leq 2^4n
$$
\end{inegalite}
\begin{proof}
En tenant compte que $n\geq 2$ et $n^{1/n} \leq 3^{1/3}$, on a:
$$
\tfrac1n\log(2^{27}\cdot 2^{n/2}\cdot n^{29/2})
=\log(2^{27/n}\cdot 2^{1/2}\cdot n^{29/2n})
\leq \log(2^{14}\cdot 3^{29/6})\leq 2^4.
$$
\end{proof}

\begin{inegalite}
\label{inega:10}
Pour $n\geq 2$ on a: 
$$
(2^{17}n^{33/4})^{2n}\leq (4n)^{17n}
$$
\end{inegalite}
\begin{proof}
$t\to f(t):=34\log(2)+(33/2)\log(t)-17\log(4t)$ est décroissante pour $t>0$ et $f(2)<0$. 
\end{proof}

\end{document}